\documentclass[12pt]{article}

\usepackage{amsmath}
\usepackage{latexsym}
\usepackage{amsfonts}
\usepackage{amssymb}
\usepackage[english]{babel}

\newcommand{\al}{\alpha}
\newcommand{\bt}{\beta}

\newcommand{\Om}{\Omega}
\newcommand{\lb}{\lambda}
\newcommand{\sg}{\sigma}
\newcommand{\essi}{\operatornamewithlimits{ess\,inf}}
\newcommand{\esss}{\operatornamewithlimits{ess\,sup}}

\newcommand{\ve}{\varepsilon}
\newcommand{\gm}{\gamma}

\newcommand{\vi}{\varphi}

\newcommand{\intl}{\int\limits}

\newcounter{theorem}

\renewcommand{\thetheorem}{\arabic{section}.\arabic{theorem}}

\newenvironment{thm}[1]{\par\addvspace{0.5cm}
    \begin{sloppypar}\refstepcounter{theorem}%
    {\bf #1 \thetheorem.}\it{}}{\end{sloppypar}}

\newcommand{\eh}{\hfill}\newlength{\sperr}

\newenvironment{theorem}{\begin{thm}{Theorem}} {\end{thm}}
\newenvironment{lemma}{\begin{thm}{Lemma}} {\end{thm}}
\newenvironment{corollary}{\begin{thm}{Corollary}} {\end{thm}}
\newenvironment{proposition}{\begin{thm}{Proposition}} {\end{thm}}

\newenvironment{defi}[1]{\par\addvspace{0.5cm}
\begin{sloppypar}\refstepcounter{theorem}%
{\bf #1 \thetheorem.}\rm{}}{\end{sloppypar}}

\newenvironment{definition}{\begin{defi}{Definition}}{\end{defi}}

\newenvironment{proof}{{\settowidth{\sperr}{\rm Proof}
\par\addvspace{0.3cm}\parbox[t]{1.3\sperr}{\rm P\eh r\eh o\eh o\eh f\eh. }%
}}{\nopagebreak\mbox{}\hfill $\Box$\par\addvspace{0.25cm}}

\begin{document}

 \centerline{\textbf{Fractional, maximal and singular operators in variable exponent Lorentz
spaces}}

\vspace{5mm} \centerline{by} \vspace{6mm}

\centerline{\textbf{Lasha Ephremidze}}

\vspace{2mm} \centerline{A.Razmadze Mathematical Institute}

\vspace{2mm} \centerline{M. Aleksidze St. 1, 380093 Tbilisi,
Georgia}

\vspace{2mm} \centerline{ \textit{ lasha@rmi.acnet.ge}}

\vspace{6mm} \centerline{\textbf{Vakhtang Kokilashvili}}

\vspace{2mm} \centerline{A.Razmadze Mathematical Institute}

\vspace{2mm} \centerline{M. Aleksidze St. 1, 380093 Tbilisi,
Georgia}

\vspace{2mm} \centerline{ \textit{ kokil@rmi.acnet.ge}}

\vspace{2mm} \centerline{and}

\vspace{6mm} \centerline{\textbf{Stefan Samko}}

\vspace{2mm} \centerline{Universidade do Algarve} \centerline{Faro
8005-139, Portugal}
 \centerline{\textit{ssamko@ualg.pt}}

\vskip+1cm

\vspace{9mm} \centerline{\textit{Dedicated to Professor Anatoly
Kilbas on the occasion of his 60th birthday}}

\centerline{\textbf{Abstract}}

\noindent{\footnotesize We introduce the Lorentz space
$\mathcal{L}^{p(\cdot), q(\cdot)}$ with variable exponents
$p(t),q(t)$ and prove the boundedness of singular integral and
fractional type operators, and corresponding ergodic operators in
these spaces. The main goal of the paper is to show that the
boundedness of these operators in the spaces
$\mathcal{L}^{p(\cdot), q(\cdot)}$ is possible without the local
log-condition on the exponents, typical for the variable exponent
Lebesgue spaces; instead the exponents $p(s)$ and $q(s)$ should
only satisfy decay conditions of log-type as $s\to 0$ and
$s\to\infty$.  To prove this, we base ourselves on the recent
progress in the problem of the validity of Hardy inequalities in
variable exponent Lebesgue spaces.

\normalsize
 \vskip+0.5cm \noindent {\em Keywords and Phrases:}
Banach function space; non-increasing rearrangement; variable
exponent; singular integral operators; fractional integral
operator, ergodic maximal function, ergodic Hilbert transform.

\vskip+0.5cm \noindent {\em AMS Classification 2000:} 42B20,
47B38, 42A50 \vskip+0.5cm

\section{Introduction}

\setcounter{equation}{0}

Nowadays the so called variable exponent analysis is a popular
topic which continues to attract many researchers, both in view of
possible applications and also because of difficulties in
investigation and existing challenging problems. This topic is
mainly focused on the Lebesgue and Sobolev spaces with variable
order of integrability and operator theory in these spaces. In
particular, various results on non-weighted and weighted
boundedness in Lebesgue spaces with variable exponents $p(x)$ have
been proved for maximal, singular and fractional type operators,
we refer to surveying papers \cite{106b}, \cite{316b},
\cite{580bd}. As is  well known, these boundedness results in the
case of a bounded open set in $\mathbb{R}^n$ hold under the
assumption that the exponent satisfies everywhere the local
log-condition
\begin{equation}\label{1}
    |p(x)-p(y)|\leq \frac{A}{\ln \frac{1}{|x-y|}},
\end{equation}
for all $x,y\in \Om $ with $|x-y|\leq \frac{1}{2}$. In the case of
unbounded sets in $\mathbb{R}^n$, it is also supposed that there
exists the limit $p(\infty)=\lim\limits_{\Om \ni x\to\infty }p(x)$
and the decay condition of log-type
\begin{equation}\label{2}
    |p(x)-p(\infty)|\leq \frac{C}{\ln (e+|x|)}
\end{equation}
is satisfied. Conditions (\ref{1})-(\ref{2}) are known to be
necessary, in terms of continuity moduli, for the boundedness of
the maximal operator in the spaces $L^{p(\cdot)}(\Om)$ with
variable exponent $p(x)$, see \cite{101ab}, \cite{479a}. Since the
known means to study singular and fractional operators in variable
exponent spaces are somehow related to the maximal operator,
assumptions (\ref{1})-(\ref{2}) are always inherited, when one
deals with those operators.

The goal of this note is to show that in the case of the Lorentz
spaces $\mathcal{L}^{p(\cdot), q(\cdot)}(\mathbb{R}^n)$, when
$p(t),q(t)$ are functions of $t\in \mathbb{R}^1_+$, the local
log-condition (\ref{1}) is no more needed for the boundedness of
the maximal operator in $\mathcal{L}^{p(\cdot),
q(\cdot)}(\mathbb{R}^n)$, we may use only  decay conditions at two
points, at $t=0$ and $t=\infty$:
\begin{equation}\label{3}
     |p(t)-p(0)|\leq \frac{C}{\ln|t|} \ \ \textrm{for} \ \     |t|\le
    \frac12,  \quad    \textrm{and} \quad    |p(t)-p(\infty)|\leq\frac{C}{\ln (e+|t|)}.
\end{equation}
We base ourselves on a recent result \cite{107e} on the validity
of the one-dimensional Hardy inequalities under
 assumptions of  type (\ref{3}).

The spaces $\mathcal{L}^{p(\cdot)}(\Om)=\mathcal{L}^{p(\cdot),
p(\cdot)}(\Om)$ have already been introduced, see \cite{321},
where the boundedness of singular and fractional type operators
was obtained under the assumption that the local log-condition
(\ref{1}) holds.  Making use of the progress for the Hardy
inequalities in \cite{107e}, we now are able to avoid that
condition and admit Lorentz spaces $\mathcal{L}^{p(\cdot),
q(\cdot)}(\Om)$.

\section{Definitions}\label{sec2}

\setcounter{equation}{0}

\subsection{On variable exponent Lebesgue spaces and Hardy operators}\label{subs21}
Let $\Omega$ be an open set in $\mathbb{R}^n$ and $\mu$ a Borel
measure on $\Om$.  Let $p(x)$ be a $\mu$-measurable function on
$\Omega$ such that $
   1\le p_-:=\essi p(x)\leq p_+:=\esss p(x)<\infty.
$ By $L^{p(\cdot)}(\Omega)$ we denote the space of measurable
functions $f(x)$ on $\Omega$ such that
$$
    \mathfrak{I}_p(f)=\int_\Omega|f(x)|^{p(x)}d\mu (x)<\infty.
$$
This is a Banach function space with respect to the norm
$$
    \|f\|_{L^{p(\cdot)}}=\inf\left\{\lambda>0: \mathfrak{I}_p\left(\frac{f}{\lb}\right)\leq 1\right\}
$$
(see e.g. \cite{146}). We refer to \cite{53b} for definition and
fundamental properties of Banach function spaces.

We denote $\frac{1}{p^\prime(x)}=1-\frac{1}{p(x)}$.

In the one-dimensional case $n=1$ we deal with the interval
$[0,\ell], \ 0<\ell \le \infty$ and the standard Lebesgue measure.
Let
$$p_-=\inf\limits_{t\in[0,\ell]}p(t), \ \ \  \ p_+=\sup\limits_{t\in[0,\ell]}p(t).$$

We will use the notation
\begin{equation}\label{classP}
\mathcal{P}_a = \{p: a <p_-\le p_+<\infty\}, \quad a\in
\mathbb{R}^1
 \end{equation}
and will be interested in the special cases of the classes
$\mathcal{P}_a$ with $a=0$ or $a=1$.
\begin{definition}\label{34}   By $\mathbb{P}([0,\ell])$ we denote the
class of functions  $p\in L^{\infty}([0,\ell]) $ such that there
exist the limits
$$p(0)=\lim\limits_{t\to 0}p(t) \quad \textrm{and} \quad p(\infty)=\lim\limits_{t\to \infty}p(t),$$
and conditions (\ref{3}) are satisfied,  the conditions at
infinity being only needed in the case $\ell =\infty.$ We also
denote
$$\mathbb{P}_a([0,\ell])= \mathbb{P}([0,\ell])\cap \mathcal{P}_a([0,\ell]).$$
\end{definition}

We recall that for $p\in \mathcal{P}_1([0,\ell])$ the H\"older
inequality
\begin{equation}\label{Holder}
\left|\intl_0^\ell u(t)v(t) dt\right|\le
k\|u\|_{L^{p(\cdot)}}\|v\|_{L^{p^\prime(\cdot)}}
 \end{equation}
holds with $k=\frac{1}{p_-}+\frac{1}{p^\prime_-}$.

 In \cite{107e} the following statement was proved.

\begin{theorem}\label{Hardy}  Let $p\in \mathbb{P}_1([0,\ell])$ and
$\al,\bt, \nu\in \mathbb{P}([0,\ell])$  and
\begin{equation}\label{chto}
 0\le \nu(0) < \frac{1}{p(0)} \ \ \ \ \textrm{and}\ \ \ \ 0\le \nu(\infty) < \frac{1}{p(\infty)}.
 \end{equation}
 Let also
$q(x)$ be any function in $ \mathbb{P}_1([0,\ell])$ such that
\begin{equation}\label{cvy7w}
\frac{1}{q(0)}=\frac{1}{p(0)}-\nu(0) \ \ \ \ \ \textrm{and}\ \ \ \
\ \frac{1}{q(\infty)}=\frac{1}{p(\infty)}-\nu(\infty).
\end{equation}
 Then the Hardy-type inequalities
\begin{equation}\label{citx1new}
\left\|t^{\al(t)+\nu(t)-1}\intl_0^t\frac{f(s)\,ds}{s^{\al(s)}}\right\|_{L^{q(\cdot)}([0,\ell])}
\le C \left\|f\right\|_{L^{p(\cdot)}([0,\ell])}
\end{equation}
\begin{equation}\label{citdox2new}
\left\|t^{\bt(t)+\nu(t)}\intl_t^\ell\frac{f(s)\,ds}{s^{\bt(s)
+1}}\right\|_{L^{q(\cdot)}([0,\ell])} \le C
\left\|f\right\|_{L^{p(\cdot)}([0,\ell])},
\end{equation}
are valid, if and only if
\begin{equation} \label{0v0dd}
\al(0) <  \frac{1}{p^\prime(0)}, \ \quad   \al(\infty)<
\frac{1}{p^\prime(\infty)}
\end{equation}
and
\begin{equation} \label{ffz}
\bt(0) >  - \frac{1}{p(0)},\ \quad
\bt(\infty)>-\frac{1}{p(\infty)},
\end{equation}
respectively (conditions at the point $\infty$ in
(\ref{chto})-(\ref{cvy7w}) and (\ref{0v0dd})-\eqref{ffz} being
only required in the case $\ell=\infty$)
\end{theorem}

\subsection{Variable exponent Lorentz spaces}\label{subs22}

In the sequel we denote $\ell=\mu \Om$ for brevity.  On the base
of the Lebesgue $L^{p(\cdot)}([0,\ell])$ we introduce now some new
Banach function spaces, variable exponent Lorentz spaces. By
$$
f^*(t)=\sup\{s\geq 0:\mu(\{ x\in\Omega:|f(x)|>s\})>t\}
$$
we denote the non-increasing rearrangement of a function $f$.
Obviously $f^\ast(t)\equiv 0$ for $t>\ell$ in case $\ell<\infty.$
\begin{definition}\label{def1} Let $p,q\in \mathcal{P}_0([0,\ell])$. By
$\mathcal{L}^{p(\cdot),q(\cdot)}(\Om)$ we denote the space of
functions $f$ on $\Om$ such that
$t^{\frac{1}{p(t)}-\frac{1}{q(t)}}f^* (t)\in
L^{q(\cdot)}([0,\ell])$, i.e.
\begin{equation}\label{4}
\mathcal{\mathfrak{I}}_{p,q}(f):=\int\limits_0^\ell
t^{\frac{q(t)}{p(t)}-1}\left|f^\ast(t)\right|^{q(t)}\,dt< \infty,
\end{equation}
and we use the notation
\begin{equation}\label{5}
\|f\|_{_{\mathcal{L}^{p,q}(\Om)}}=\inf\left\{\lb>0:
\mathcal{\mathfrak{I}}_{p,q}\left(\frac{f}{\lb}\right)\le 1
\right\}
 = \left\|t^{\frac{1}{p(t)}-\frac{1}{q(t)}}f^*
 (t)\right\|_{L^{q(\cdot)}([0,\ell])}.
\end{equation}
\end{definition}

 It is easy to see that in the case $p\in \mathbb{P}_0([0,\ell]), q \in
\mathbb{P}_1([0,\ell])$, condition (\ref{4}) is equivalent to the
condition
\begin{equation}\label{6}
\int\limits_0^1
t^{\frac{q(0)}{p(0)}-1}\left|f^\ast(t)\right|^{q(t)}\,dt +
\int\limits_1^\infty
t^{\frac{q(\infty)}{p(\infty)}-1}\left|f^\ast(t)\right|^{q(t)}\,dt<
\infty,
\end{equation}
the latter being written for the case $\ell=\infty.$ In the case
$\ell<\infty$, only the term $\int\limits_0^\ell
t^{\frac{q(0)}{p(0)}-1}\left|f^\ast(t)\right|^{q(t)}\,dt$  should
be considered.

 Let
$$
    f^{**}(t)=\frac{1}{t} \int_0^tf^*(s)ds, \ \ \ \ \ \ \ \ \ f^*(t)\leq f^{**}(t).
$$

We can  introduce the norm
\begin{equation}\label{5}
\|f\|^1_{_{\mathcal{L}^{p,q}(\Om)}}
 = \left\|t^{\frac{1}{p(t)}-\frac{1}{q(t)}}f^{**} (t)\right\|_{L^{q(\cdot)}([0,\ell])},
\end{equation}
so that $$\|f\|_{_{\mathcal{L}^{p,q}(\Om)}}\le
\|f\|^1_{_{\mathcal{L}^{p,q}(\Om)}}.$$ The equivalence of (2.10)
and (2.12) is characterized in the following theorem.
\begin{theorem}\label{lem}
Let $p\in\mathbb{P}_0([0,\ell]), q \in\mathbb{P}_1([0,\ell])$.
Then the inequality $\|f\|^1_{_{\mathcal{L}^{p,q}(\Om)}}\le C
\|f\|_{_{\mathcal{L}^{p,q}(\Om)}}$  with a constant $C>0$ not
depending on $f$, holds if and only if $p(0)>1$ and, in case the
$\ell=|\Om|=\infty$, also $p(\infty)>1$.
\end{theorem}
\begin{proof}
Indeed, the inequality $\|f\|^1_{_{\mathcal{L}^{p,q}(\Om)}}\le C
\|f\|_{_{\mathcal{L}^{p,q}(\Om)}}$ is nothing else but the
boundedness in $L^{q(\cdot)}([0,\ell])$ of the Hardy operator
$$t^{\frac{1}{p(t)}-\frac{1}{q(t)}-1}\intl_0^t \frac{f(s)ds}{s^{\frac{1}{p(s)}-\frac{1}{q(s)}}}. $$
By Theorem \ref{Hardy}, this boundedness is valid if and only if
the values of $\frac{1}{p(t)}-\frac{1}{q(t)}$ at the points $t=0$
and $t=\infty$ are less than those of $\frac{1}{q^\prime(t)}$ at
these points, respectively. This gives conditions $p(0)>1,
p(\infty)>1$.
\end{proof}

Note that in all the statements in the sequel, all the conditions
imposed on $p(t),q(t)$ at the point $t=\infty$ should be omitted
in the case where $|\Om|< \infty$.

\vspace{2mm} In accordance with Theorem \ref{lem}, in the sequel
we consider the space $\mathcal{L}^{p(\cdot),q(\cdot)}(\Om)$ under
the following assumptions on $p(\cdot)$ and $q(\cdot)$:
\begin{equation}\label{suppose}
p,q \in\mathbb{P}_1([0,\ell]) \quad \textrm{and} \quad p(0)>1, \
p(\infty)>1.
\end{equation}

\subsection{Basic properties of the spaces $\mathcal{L}^{p,q}(\Om)$} \label{subs23}

We refer to  \cite{53b} for the notion of  Banach function space
(BFS) and rearrangement invariant norms, but recall the following
basic definition, where $M(\Omega,\mu)$ denotes the set of all
$\mu$-measurable functions on $\Om$.

\begin{definition}
A normed linear space $X=(X(\Omega,\mu),\|\;\|_X)$ is called a
Banach function space, if the
following conditions are satisfied: \\
i$)$ the norm $\|f\|_X$ is defined for all $f\in M(\Omega,\mu)$;
\\
ii$)$ $\|f\|_X=0$ if and only if $f(x)=0$ $\mu$-a.e. on $\Omega$;
\\
iii$)$ $\|f\|_X=\big\||f|\big\|_X$ for all $f\in X$;
\\
iv$)$ for every $Q\subset \Omega$ with $\mu Q<\infty$ we have
$\|\chi_Q\|_X<\infty$;
\\
v$)$ if $f_n\in M(\Omega,\mu)$, $n=1,2,\dots$ and $f_n \nearrow f$
$\mu$-a.e. on $\Omega$, then $
    \|f_n\|_X \nearrow\|f\|_X;\\
$ vi$)$ if $f$, $g\in M(\Omega,\mu)$ and $0\leq f(x)\leq g(x)$
$\mu$-a.e. on $\Omega$, then $
    \|f\|_X \leq\|g\|_X;
\\ $  vii$)$ given $Q\subset \Omega$ with $\mu Q<\infty$, there exists a constant $c_Q$ such that for
all $f\in X$, $
    \int_Q|f(x)|d\mu\leq c_Q\|f\|_X.
$
\end{definition}

 \vspace{3mm}In particular, the following statement is known
(\textrm{\cite{53a}, p.61}).
\begin{proposition}\label{prop}
Let $(X,\mu)$ be an arbitrary totally $\sg$-finite measure space
and $\lb(g)$ a rearrangement-invariant norm over
$(\mathbb{R}^1,m)$. Then the functionional $\rho(f)$ defined on
functions $f$ in $(X,\mu)$ by $\rho(f)=\lb(f^*)$ is a
rearrangement-invariant norm on  $(X,\mu)$.
\end{proposition}

\begin{lemma}\label{lem1}
Let $p,q\in \mathcal{P}_1(\Om)$. Then the dual space
$\left(\mathcal{L}^{p(\cdot),q(\cdot)}(\Om)\right)^*$ is
$\mathcal{L}^{p^\prime(\cdot),q^\prime(\cdot)}(\Om)$.
\end{lemma}
\begin{theorem}\label{theor}
Under conditions \eqref{suppose}, the space
$\mathcal{L}^{p,q}(\Om)$ is a Banach function space.
\end{theorem}
\begin{proof}
To state that both $\|f\|_{_{\mathcal{L}^{p,q}(\Om)}}$ and
$\|f\|^1_{_{\mathcal{L}^{p,q}(\Om)}}$ are norms, it suffices to
refer to Proposition \ref{prop}. (The  triangle inequality for the
norm $\|f\|^1_{_{\mathcal{L}^{p,q}(\Om)}}$ follows from the
inequality $
    (f+g)^{**}(t)\leq f^{**}(t)+g^{**}(t)$, see e.g. \cite{235}, Section 2, or \cite{53a}, p. 54).
     The other requirements to the definition of BFS
  easily follow from properties of non-increasing rearrangements $f^*$ and
properties of the spaces $L^{p(\cdot)}$. For example, iv) is valid
since for $0\leq f_n \nearrow f$ we have $f_n^* \nearrow f^*$ (see
e.g. \cite{647}, Lemma 3.5, Chapter 5). Then
$$
\|f_n\|_{_{\mathcal{L}^{p,q}(\Om)}}=
\left\|t^{\frac{1}{p(t)}-\frac{1}{q(t)}}f^*_n\right\|_{L^{q(\cdot)}([0,\ell])}\nearrow
\|f\|_{_{\mathcal{L}^{p,q}(\Om)}}
$$
by the property of the space  $L^{q(\cdot)}$. To check vii), we
make use of the H\"{o}lder inequality \eqref{Holder} for
$L^{q(\cdot)}$ with $u(t)=t^{\frac{1}{q(t)}-\frac{1}{p(t)}}$ and
$v(t)=t^{\frac{1}{p(t)}-\frac{1}{q(t)}}f^\ast (t)$ and  get
$$
    \int_Q|f(x)|dx=\int_0^{\mu Q}f^*(t)dt
    \leq \|u\|_{L^{q^\prime(\cdot)}([0,\ell])} \|f\|_{L^{p(\cdot), q(\cdot)}(\Om)}\leq c_Q\|f\|_{L^{p(\cdot),q(\cdot)}(\Om)}
$$
with $c_Q=\|u\|_{{L^{q^\prime(\cdot)}}([0,\ell])}<\infty$ because
$\|u\|_{L^{q^\prime(\cdot)}([0,\ell])}<\infty \Longleftrightarrow
\mathfrak{I}_q(u)<\infty$, the latter being valid under the
condition $p(0)>1$, which was assumed.
\end{proof}

Let $w(t)$ be a nonnegative weight function defined on $[0,\ell]$.

\begin{definition}\label{def2} We define the  weighted Lorentz space $\mathcal{L}^{p(\cdot),q(\cdot)}_w(\Om)$
with the weight $w$ defined on $[0,\ell]$, as the subset of
functions in $M(\Omega,\mu)$ such that
\begin{equation}\label{eq}
   \|f\|_{\mathcal{L}^{p(\cdot),q(\cdot)}_w(\Om)}=
   \left\|w(t)t^{\frac{1}{p(t)}-\frac{1}{q(t)}}f^{*}(t)\right\|_{L^{q(\cdot)}(\Om)}<\infty.
   \end{equation}
\end{definition}

Let also
\begin{equation}\label{norma}
   \|f\|^1_{\mathcal{L}^{p(\cdot),q(\cdot)}_w(\Om)}=
   \left\|w(t)t^{\frac{1}{p(t)}-\frac{1}{q(t)}}f^{\ast*}(t)\right\|_{L^{q(\cdot)}(\Om)}.
   \end{equation}

 In the next lemma we suppose that $\gm(t)$ is a measurable bounded function on $[0,\ell]$
having the limit $\gm(0)=\lim\limits_{t\to 0+}\gamma(t)$, and, in
the case $\ell=\infty$, also having the limit
$\gm(\infty)=\lim\limits_{t\to +\infty}\gamma(t)$ and satisfying
the conditions
\begin{equation}\label{beta}
|\gm(t)-\gm(0)|\le \frac{C}{\ln\frac{1}{t}}, \quad 0<t<
\frac{1}{2} \quad \textrm{and} \quad |\gm(t)-\gm(\infty)|\le
\frac{C}{\ln(e+t)}.
\end{equation}

\begin{lemma}\label{lem2} Let the conditions in \eqref{suppose} be satisfied and
let  $w(t)=t^{\gm(t)}$, where $\gm(t)$ satisfies conditions
\eqref{beta} and
$$\gamma(0)<\frac{1}{p^\prime(0)} \quad  \textrm{and} \quad \gamma(\infty)<\frac{1}{p^\prime(\infty)}.$$
Then $\|f\|_{\mathcal{L}^{p(\cdot),q(\cdot)}_w(\Om)}\le
\|f\|^1_{\mathcal{L}^{p(\cdot),q(\cdot)}_w(\Om)} \le C
\|f\|_{\mathcal{L}^{p(\cdot),q(\cdot)}_w(\Om)},$
 where $C>0$ does not depend on $f$.
\end{lemma}
\begin{proof}
The left hand side inequality is trivial, the right-hand side one
follows from Theorem \ref{Hardy}.

\end{proof}

In the next theorem we use the notation
$$\mathcal{L}^{p(\cdot)}_{loc}([0,\ell])= \left\{f: f\in \mathcal{L}^{p(\cdot)}([0,\ell_1])
 \ \textrm{for all} \ \ell_1<\ell
\right\}.$$
\begin{theorem}\label{theorw}
Under the condition
\begin{equation}\label{loc}
\frac{t^{\frac{1}{q(0)}-\frac{1}{p(0)}}}{w(t)}\in
\mathcal{L}^{q^\prime(\cdot)}_{loc}([0,\ell]),
\end{equation}
the space $\mathcal{L}_w^{p(\cdot), q(\cdot)}(\Om)$ is a Banach
function space with respect to the norm
$\|f\|^1_{\mathcal{L}^{p(\cdot),q(\cdot)}_w(\Om)}$.
\end{theorem}

The proof is similar to that of Theorem \ref{theor}.

\section{On classical operators in the space $\mathcal{L}^{p(\cdot),q(\cdot)}_w(\Om)$}

 Let
\begin{equation}\label{max}
\mathcal{M}f(x) = \sup\limits_{r>0}\frac{1}{\mu
B(x,r)}\intl_{\Om\cap B(x,r)} |f(y)| d\mu (y), \quad x\in \Om,
\end{equation}
be the Hardy-Littlewood maximal function.
\begin{theorem}\label{theormax}
Let $p$ and $q$ satisfy assumptions (\ref{suppose}).  Then the
maximal operator is bounded in the space
$\mathcal{L}^{p(\cdot),q(\cdot)}_w(\Om)$ with the weight
\begin{equation}\label{weight}
w(t)=t^{\gm(t)}, \ \ \gm\in \mathbb{P}([0,\ell]),
\end{equation}
if
\begin{equation}\label{gm}
\gm(0)<\cfrac{1}{p^\prime(0)} \quad \textrm{and} \quad
\gm(\infty)<\cfrac{1}{p^\prime(\infty)} \quad \textrm{(the latter
in the case}  \quad \mu(\Om)=\infty\textrm{)}.
\end{equation}
\end{theorem}
\begin{proof}
As is known,
\begin{equation}\label{maxpointwise}
(\mathcal{M}f)^\ast(t)\le C f^{\ast\ast}(t),
\end{equation}
see for instance \cite{53b}, p.122. Therefore,
\begin{equation}\label{estimation}
\|\mathcal{M}f\|_{\mathcal{L}_w^{p(\cdot),q(\cdot)}(\Om)}=
\left\|t^{\gm(t)+\frac{1}{p(t)}-\frac{1}{q(t)}}
(\mathcal{M}f)^\ast\right\|_{L^{q(\cdot)}([0,\ell])}\le C
\left\|t^{\gm(t)+\frac{1}{p(t)}-\frac{1}{q(t)}}
f^{\ast\ast}\right\|_{L^{q(\cdot)}([0,\ell])}
\end{equation}
and then the result follows by Theorem \ref{Hardy}.
\end{proof}

As is known, the identity approximations
$$A_\ve f(x)= \frac{1}{\ve^n}\intl_{\mathbb{R}n} a\left(\frac{x-y}{\ve}\right)f(y),$$
where $ \int_{\mathbb{R}n} a(y)\, dy =1$ and   $a(x)$ has a radial
decreasing integrable majorant, are dominated by the maximal
operator:
\begin{equation}\label{domination}
\left|A_\ve f(x)\right|\le C \mathcal{M}f (x), \quad f\in
L^p(\mathbb{R}^n),\  1\le p\le \infty,
\end{equation}
with an absolute constant $C>0$ not depending on $x$ and $\ve,$
see \cite{642a}. In particular, the Poisson integral
$$\mathbf{P}_y f (x)=\intl_{\mathbb{R}^n}P\left(x-\xi,y\right)f(\xi)d\xi, \ \quad P(x,y)=
\frac{c_n y}{\left(|x|^2+y^2\right)^\frac{n+1}{2}}, \ \ y>0 $$ is
uniformly in $y$ dominated by the maximal function.
 Under assumptions of Theorem \ref{theormax} we
have
\begin{equation}\label{ask}
\mathcal{L}^{p(\cdot),q(\cdot)}_w(\Om) \subset
L^1(\Om)+L^\infty(\Om).
\end{equation}
 So we make use of
\eqref{domination} and arrive at the following corollary.

 \begin{corollary}\label{cor} Under the assumptions of Theorem \ref{theormax}, the sublinear operator
$$\sup\limits_{\ve >0}\left|A_\ve f(x)\right|,$$
where $A_\ve f$ is an   identity
 approximation with kernel  admitting radial decreasing integrable majorant, is
 bounded in the space $\mathcal{L}^{p(\cdot),q(\cdot)}_w(\Om)$;  in particular the operator
$\sup\limits_{y >0}\left|\mathbf{P}_y f(x)\right|$ is bounded in
this space.
 \end{corollary}

 Next we consider in $\mathcal{L}_w^{p(\cdot), q(\cdot)}(\Om)$ convolution operators
$$k\ast f(x)=\intl_{\mathbb{R}^n}k(x-y)f(y)d\mu(y).$$
We will also treat  their particular cases, the Riesz potential
operator and Calderon-Zygmund singular operators, which for
generality we will consider over an open set $\Om\subseteq
\mathbb{R}^n$:
$$
    I^\alpha f(x)=\int_{\Omega}\frac{f(y)}{|x-y|^{n-\alpha}}\;d\mu(y),\;\;\;x\in \Omega,\;\;\;0<\alpha<n.
$$
and
$$
Kf(x)=\int_{\Om}\frac{A(x-y)}{|x-y|^n}f(y)\;d\mu(y),\;\;\;x\in
\Om,
$$
 where $A$ is an odd function on $\mathbb{R}^n$,
homogeneous of degree $0$ and satisfying the Dini condition on the
unit sphere $\mathbb{S}^{n-1}$:
$$
    \int_0^2\frac{\omega(A,\delta)}{\delta}\;d\delta<\infty, \;\;\;\mbox{where}\;\;\;
    \omega(k,\delta)=\sup_{x,y\in S^{n-1},
                      |x-y|\leq \delta}|A(x)-A(y)|.
$$
The operators $K$  include as particular cases, the Hilbert
transform \ $(n=1,\ k(x)=\frac{x}{|x|})$ and the Riesz transforms
$(n\geq 2, \  k(x)=\frac{x_j}{|x|}$, $j=1,\dots,n)$.

There are known the following pointwise estimates of those
classical operators via decreasing rearrangements:
\begin{equation}\label{ONeil}
(k\ast f)^\ast (t)\le k^{\ast\ast}(t)\intl_0^t f(s)ds +
\intl_t^\infty k^\ast(s)f^\ast(s)ds,
\end{equation}
see \cite{455a}, and its particular case
\begin{equation}\label{ONeilRiesz}
    (I^\alpha f)^*(t)\!\leq\! c\bigg(t^{-1+\alpha/n}\int_0^tf^*(s)ds\!+\!\int_t^\ell f^*(s)
    s^{-1+\alpha/n}ds\bigg), \quad \ell= \mu(\Om).
\end{equation}
A similar estimate holds for the singular operator $K$
\begin{equation}\label{12}
   (Kf)^*(t)\leq c\bigg(\frac{1}{t}\int_0^tf^*(s)ds+\int_t^\ell \frac {f^*(s)}{s}\,ds\bigg),
        \; \quad \ell=\mu \Om,
\end{equation}
 see   \cite{53a}.

\begin{theorem}\label{theorsingular}
 Let $p$ and $q$ satisfy assumptions (\ref{suppose}).
 Then the operator $K$ is bounded in the space $\mathcal{L}^{p(\cdot),q(\cdot)}_w(\Om)$ with the weight
\eqref{weight} under conditions \eqref{gm}.

\end{theorem}

\begin{proof} The proof is obtained similarly to \eqref{estimation} from the
pointwise estimate \eqref{12}
  and Theorem \ref{Hardy}.
\end{proof}

\begin{theorem}\label{Riesz}
Let $0<\al<n$,  $p$ and $q$ satisfy assumptions (\ref{suppose})
and $p_+<\frac{n}{\al}$.
 Then the operator $I^\al$ is bounded from  the space $\mathcal{L}^{p(\cdot),q(\cdot)}_w(\Om)$ with the weight
\eqref{weight} into the space
$\mathcal{L}^{p_\al(\cdot),q(\cdot)}_w(\Om)$ where
$\frac{1}{p_\al(t)}=\frac{1}{p(t)}-\frac{\al}{n}$, if
\begin{equation}\label{condgm}
\frac{\al}{n}-\frac{1}{p(0)}<\gm(0)<\cfrac{1}{p^\prime(0)} \quad
\textrm{and} \quad \frac{\al}{n}-\frac{1}{p(\infty)} <
\gm(\infty)<\cfrac{1}{p^\prime(\infty)},
\end{equation}
the condition at infinity being needed in  the case $
\mu(\Om)=\infty$.
\end{theorem}

\begin{proof}  We have
$$ \left\|I^\al f\right\|_{\mathcal{L}^{p_\al(\cdot),q(\cdot)}_w(\Om)} =
\left\|t^{\gm(t)+\frac{1}{p_\al(t)}-\frac{1}{q(t)}}\left(I^\al
f\right)^\ast(t) \right\|_{L^{q(\cdot)}([0,\ell])}.
$$
Then by  \eqref{ONeilRiesz}
$$\left\|I^\al f\right\|_{\mathcal{L}^{p_\al(\cdot),q(\cdot)}_w(\Om)}\le c(A+B),$$
where
$$A=\left\|t^{\lb(t)-1}\intl_0^t\frac{\vi(s)ds}{s^{\lb(s)}}\right\|_{L^{q(\cdot)}([0,\ell])}, \quad
B=\left\|t^{\lb(t)-\frac{\al}{n}}\intl_0^t\frac{\vi(s)ds}{s^{\lb(s)-\frac{\al}{n}}+1}
\right\|_{L^{q(\cdot)}([0,\ell])}$$ and
$\lb(t)=\gm(t)+\frac{1}{p(t)}-\frac{1}{q(t)}$ and
$\vi(t)=t^{\lb(t)}f^\ast(t)\in L^{q(\cdot)}([0,\ell])$. It remains
to make use of Theorem \ref{Hardy}.
\end{proof}

Since the fractional maximal function
$$\mathcal{M}^{\alpha}f(x)=\sup\limits_{r>0}\frac{1}{|B(x,r)|^{1-\frac{\alpha}{n}}}
\int\limits_{B(x,r)\cap \Om }|f(y)|dy, \ \quad  0<\al<n,
$$
is dominated by fractional integral: $\mathcal{M}^{\alpha} f(x)\le
c \,I^{\alpha} (|f|)(x),$ from Theorem \ref{Riesz} we get the
following corollary,
\begin{corollary}\label{cor1}
Under the assumptions of Theorem \ref{Riesz}, the operator
$\mathcal{M}^{\alpha}$ is bounded from the space
$\mathcal{L}^{p(\cdot),q(\cdot)}_w(\Om)$  into the space
$\mathcal{L}^{p_\al(\cdot),q(\cdot)}_w(\Om)$.
\end{corollary}

 \section{On the ergodic maximal function and the ergodic Hilbert transform in variable exponent Lorentz spaces}

Let $(T_\tau)_{\tau\in \mathbb{R}}$ be an ergodic flow of
measure-preserving transformations on a $\sigma$-finite measure
space $(X,\mu)$, and let $\mathbf{M}f$ and $\mathbb{H}f$, $f\in
L(X)$, be the ergodic maximal function and the ergodic Hilbert
transform, respectively, (see \cite{Pet})
$$
\mathbf{M}f(x)=\sup_{a>0}\frac1a\int_0^a|f(T_\tau x)|\,d\tau\;\;
\text{ and }\;\;
\mathbb{H}f(x)=\lim_{\delta\to0+}\;\int_{\{\delta\leq|t|\leq
1/\delta\}} \frac{f(T_{\tau}x)}{\tau}\,d\tau.
$$

The estimations (3.21) and (3.27) hold for operators $\mathbf{M}$
and $\mathbb{H}$, respectively, as well. Namely,
\begin{equation}
(\mathbf{M}f)^*(t)\leq f^{**}(t)
\end{equation}
can be obtained as in the discrete case (see \cite{Eph1}; Ineq.
(2)) since only the weak $(1,1)$ type inequality,
$\mu\{\mathbf{M}f)^*>\lambda\}\leq\frac1{\lambda}\int_{\{\mathbf{M}f)^*>\lambda\}}f\,d\mu$,
is used to prove (4.29) in the discrete case which holds for the
continuous case too with equation sign (see \cite{Pet}, p. 76),
and the inequality
\begin{equation}
(\mathbb{H}f)^*(t)\leq
c\bigg(\frac{1}{t}\int_0^tf^*(s)ds+\int_t^\ell \frac
{f^*(s)}{s}\,ds\bigg),
        \; \quad \ell=\mu(X),
\end{equation}
can be proved using the generalization of the Stein-Weiss theorem
for the ergodic Hilbert transform (see \cite{Eph2}, \cite{Eph3}):
\begin{equation}
\mu\{|\mathbb{H}(\mathbf{1}_E)|>\lambda=\begin{cases}
\Psi_{\mu(E)}(\lambda) \text{ when } \mu(X)=\infty\\
\Phi_{\mu(E)}(\lambda) \text{ when } \mu(X)<\infty\,,
\end{cases}
\end{equation}
where $E\subset X$ is any measurable subset, and
$$
\Psi_\xi(\lambda)=\frac{2\xi}{\sinh \lambda}\;\;\;\text{ and
}\Phi_\xi(\lambda)=\frac{2\mu(X)}\pi\,\arctan \left(\frac
{\sin(\pi\xi/\mu(X))}{\sinh\lambda}\right)\,.
$$
Indeed, if $h$ is a measurable function with strictly decreasing
continuous distribution function $\mathcal{D}_h$, then
$h^*(t)=\mathcal{D}_h^{-1}(t)$. Hence it follows from (4.31) that
$$
\big(\mathbb{H}(\mathbf{1}_E)\big)^*(t)=
\begin{cases}
\Psi_{\mu(E)}^{-1}(t) \text{ when } \mu(X)=\infty\;\text{ and } 0<t<\infty\\
\Phi_{\mu(E)}^{-1}(t) \text{ when } \mu(X)<\infty\;\text{ and } 0<t<\mu(X)\\
0 \text{ when }\mu(X)<\infty\;\text{ and } t\geq\mu(X)
\end{cases}
$$
Observe that
$$
\Psi_{\mu(E)}^{-1}(t)=\sinh^{-1}\left(\frac\xi
t\right)\;\;\;\text{ and }
\Phi_{\mu(E)}^{-1}(t)=\sinh^{-1}\left(\frac{\sin(\pi\xi/\mu(X))}{\tan(\pi
t/2\mu(X))}\right)\,.
$$
The function $\sinh^{-1}$ is increasing, and if we use simple
relations between the trigonometric functions $\sin x<x$,
$0<x<\pi$ and $\tan t>t$, $0<t< \frac {\pi}2$, then we get for
each $\mu(E)<\mu(X)$ and $t>0$,
\begin{equation}
(\mathbb{H}(\mathbf{1}_E))^*(t)\leq \frac
1{\pi}\sinh^{-1}\left(\frac{2\mu(E)}t\right)
\end{equation}

The rest of the proof of (4.30) is the same as for the usual
Hilbert transform case (see \cite{Pet}, pp.134-137).

As in previous sections, depending on estimations (4.29) and
(4.30), one can prove the following

\begin{theorem}\label{theorsingular}
 Let $p$ and $q$ satisfy assumptions (\ref{suppose}).
 Then the ergodic maximal operator and the ergodic Hilbert transform are bounded in the space $\mathcal{L}^{p(\cdot),q(\cdot)}_w(\Om)$ with the weight
\eqref{weight} under conditions \eqref{gm}.

\end{theorem}

  \vspace{6mm} \textbf{Acknowledgements}.  This work was made under the project
"Variable Exponent Analysis" supported by INTAS grant
 Nr.06-1000017-8792, the first two authors were also supported by the grant GNSF/ST07/3-169.

\def\ocirc#1{\ifmmode\setbox0=\hbox{$#1$}\dimen0=\ht0 \advance\dimen0
  by1pt\rlap{\hbox to\wd0{\hss\raise\dimen0
  \hbox{\hskip.2em$\scriptscriptstyle\circ$}\hss}}#1\else {\accent"17 #1}\fi}


\end{document}